\documentclass[10pt]{article}

\usepackage{mathptmx,amsmath,amscd,amssymb,amsthm,xspace}
\usepackage[all,tips]{xy}
\usepackage[dvips]{graphicx}
\usepackage{graphics,epsfig,wrapfig,verbatim,syntonly}
\usepackage{hyperref,amssymb,color, url,fancyhdr}
\usepackage{verbatim}
\usepackage{syntonly}
\usepackage{hyperref}
\usepackage{tikz}
\usepackage{amsfonts, amsmath, amsopn, amssymb, latexsym, graphics, graphicx, enumerate,hyperref}
\usepackage{changepage}

\theoremstyle{plain}
\newtheorem{thm}{Theorem}[section]

\newtheorem{prop}[thm]{Proposition}
\newtheorem{lemma}[thm]{Lemma}

\theoremstyle{definition}

\DeclareMathOperator{\Isom}{Isom}

\newcommand{\wt}{\widetilde}

\newcommand{\innp}[1]{\left< #1 \right>}
\newcommand{\abs}[1]{\left\vert#1\right\vert}

\newcommand{\set}[1]{\left\{#1\right\}}

\newcommand{\su}{\subset}

\newcommand{\lra}{\longrightarrow}

\newcommand{\B}[1]{\ensuremath{\mathbf{#1}}}
\newcommand{\BB}[1]{\ensuremath{\mathbb{#1}}}
\newcommand{\Cal}[1]{\ensuremath{\mathcal{#1}}}
\newcommand{\Fr}[1]{\ensuremath{\mathfrak{#1}}}

\newcommand{\N}{\ensuremath{\BB{N}}}
\newcommand{\Q}{\ensuremath{\BB{Q}}}
\newcommand{\R}{\ensuremath{\BB{R}}}
\newcommand{\Z}{\ensuremath{\BB{Z}}}

\newcommand{\F}{\ensuremath{\BB{F}}}

\usepackage{fancyhdr}

\let\oldmarginpar\marginpar
\renewcommand\marginpar[1]{\-\oldmarginpar[\raggedleft\footnotesize #1]%
{\raggedright\footnotesize #1}}

\begin{document}


\title{\textbf{Constructing Geometrically \\ Equivalent Hyperbolic Orbifolds}}
\author{D.~B.~McReynolds\thanks{Purdue University, West Lafayette, IN. E-mail \tt{dmcreyno@purdue.edu}},~~Jeffrey S. Meyer\thanks{University of Oklahoma, Norman, OK. E-mail: \tt{jmeyer@math.ou.edu}}, and Matthew Stover\thanks{Temple University, Philadelphia, PA. E-mail: \tt{mstover@temple.edu}}}
\maketitle


\begin{abstract}
\noindent In this paper, we construct families of nonisometric hyperbolic orbifolds that contain the same isometry classes of nonflat totally geodesic subspaces. The main tool is a variant of the well-known Sunada method for constructing length-isospectral Riemannian manifolds that handles totally geodesic submanifolds of multiple codimensions simultaneously. 
\end{abstract}

\section{Introduction}

Classical spectra like the eigenvalue spectrum of the Laplace--Beltrami operator or the primitive geodesic length spectrum have played an important role in dynamics, geometry, and representation theory. In this paper, we continue the investigation of higher dimensional spectra that encode the geometry of the nonflat totally geodesic submanifolds of a fixed complete, finite volume, Riemannian manifold $M$. We will refer to the set of such submanifolds, counted with multiplicity, as the \textbf{geometric spectrum}.

To construct our examples, we restrict ourselves to closed arithmetic locally symmetric orbifolds, where recent work shows that the geometric spectrum, when nonempty, carries much information. In \cite{McR}, it was shown that if $M_1,M_2$ are arithmetic hyperbolic 3--manifolds with the same geometric spectrum, provided the geometric spectrum is nonempty, then $M_1$ and $M_2$ are commensurable. For higher dimensions, \cite[Thm C]{Meyer} proved that if $M_1$ and $M_2$ are standard arithmetic hyperbolic $m$--manifolds (see \S \ref{section:overview}) with the same geometric spectrum, then $M_1$ and $M_2$ are commensurable. It is well-known that the geometric spectrum of a standard arithmetic hyperbolic $m$--manifold is nonempty with representatives in every possible proper codimension.

For any finite volume, hyperbolic $3$--manifold $M$, there exist infinitely many pairs of nonisometric finite covers $(M_j,N_j)$ of $M$ such that $M_j, N_j$ have the same totally geodesic surfaces \cite{McR}. This has two parts. First, there are infinitely many pairs of finite covers $\{M_j',N_j'\}$ with the same geometric spectrum. It is a feature of this construction that $\mathrm{Vol}(M_j') = \mathrm{Vol}(N_j')$, though we know no general reason why that must hold. Secondly, there exist infinitely many pairs $\{M_j, N_j\}$ with the same set of totally geodesic surfaces (i.e., without multiplicity) such that $\mathrm{Vol}(M_j)/\mathrm{Vol}(N_j)$ is unbounded.

The main result of this article is the generalization of the above covering constructions to higher dimensional hyperbolic manifolds. We utilize a variant of the well-known Sunada method for producing length-isospectral Riemannian manifolds that allows one to handle totally geodesic submanifolds of varying codimensions. The case of totally geodesic subsurfaces of a hyperbolic $3$--manifold is handled by \cite{McR}, and the challenge we overcome is to care for all codimensions simultaneously. 

Define the \textbf{totally geodesic spectrum} of a locally symmetric Riemannian orbifold $M$ to be the set
\begin{equation}\label{defqtgm1}
\mathcal{TG}(M)=\left\{ \parbox{2.6in}{\begin{center}Isometry classes of orientable nonflat finite volume totally geodesic subspaces $X \subset M$ with multiplicity $m_X$ \end{center}}\right\} = \set{(X_j,m_{X_j})}.
\end{equation}
We say that $M_1$ and $M_2$ are \textbf{geometrically isospectral} if $\mathcal{TG}(M_1)=\mathcal{TG}(M_2)$. The \textbf{totally geodesic set} of a locally symmetric, Riemannian orbifold is
\begin{equation}\label{defqtgm}
\mathrm{TG}(M)=\left\{ \parbox{2.5in}{\begin{center}Isometry classes of orientable nonflat finite volume totally geodesic subspaces $X \subset M$ \end{center}}\right\} = \set{(X_j)}.
\end{equation}
We say that $M_1,M_2$ are \textbf{geometrically equivalent} if $\mathrm{TG}(M_1) = \mathrm{TG}(M_2)$. 

\begin{thm}\label{thmA}
For every commensurability class $\mathcal{C}$ of closed arithmetic hyperbolic $m$--orbifolds with $m\ge 3$, we have the following:
\begin{itemize}
\item[(a)] For each $M \in \mathcal{C}$, there exist nonisometric finite covers $M',N'$ of $M$ such that $\mathcal{TG}(M') = \mathcal{TG}(N')$.
\item[(b)] For each $M \in \mathcal{C}$, there exist infinitely many pairs of nonisometric, finite covers $(M_j,N_j)$ of $M$ such that:
\begin{itemize}
\item[(i)] For all $j$, $\mathrm{TG}(M_j) = \mathrm{TG}(N_j)$.
\item[(ii)] The ratio $\mathrm{Vol}(M_j) / \mathrm{Vol}(N_j)$ is unbounded.
\end{itemize}
\end{itemize}
\end{thm}

The orientability condition in \eqref{defqtgm} is a matter of taste, as a small modification of our methods allows for nonorientable geodesic subspaces. Our methods can produce examples modeled on other symmetric spaces of noncompact type, but the technicalities would obscure the basic ideas behind our construction, which is general enough to highlight the basic procedure (see Theorem \ref{thmA-Gen} for a generalization of Theorem \ref{thmA}). 

{\bf Acknowledgments}. The authors acknowledge support from U.S. National Science Foundation grants DMS 1107452, 1107263, 1107367 "RNMS: GEometric structures And Representation varieties" (the GEAR Network). The first author was supported by NSF grants DMS-1105710 and DMS-1408458. The third author was supported by the National Science Foundation under Grant Number NSF 1361000.

\section{Notation and Overview}\label{section:overview}

In this section, we outline the construction of the covers required to prove our main results. Before providing this outline, we briefly set some notation and terminology that will be used throughout the article.

\subsection{Preliminaries}\label{SS:Prelim}

A finite volume hyperbolic $m$--manifold $M$ is \textbf{arithmetic} if its fundamental group $\Gamma=\pi_1(M)$ has a commensurator $\mathrm{Comm(\Gamma)}=\{g\in \mathrm{Isom}^+(\mathbb{H}^m)\ | \ \Gamma, g\Gamma g^{-1} \mbox{ are commensurable}\}$ that is dense in $\mathrm{Isom}^+(\mathbb{H}^m)$ (see \cite[16.3.3]{Witte}). The subclass of these that exhibit the richest collections of totally geodesic submanifolds are the so-called standard arithmetic manifolds, which we now describe.

Throughout this paper, $k$ denotes a number field, $\mathcal{O}_k$ its ring of integers, and $q$ a nondegenerate quadratic form over $k$. For a prime ideal $\mathfrak{p}$ of $\mathcal{O}_k$, let $k_{\mathfrak{p}}$ denote the localization of $k$ at $\mathfrak{p}$ and $\Cal{O}_{\Fr{p}}$ is its ring of integers. Call $(k, q)$ an \textbf{admissible hyperbolic pair} when $k$ is totally real and $q$ is positive definite at all but one real place of $k$, at which it has signature $(m,1)$. Set $\mathbf{G}=\mathbf{SO}(q)$, fix a $k$--rational embedding $\iota\colon\mathbf{G}\to \mathbf{GL}_d$, and define $\mathbf{G}(\Cal{O}_k)= \iota^{-1}(\iota(\mathbf{G}(k)) \cap \mathbf{GL}_d(\Cal{O}_k))$. Since the $k$--isomorphism class of $\mathbf{G}$ is independent of the similarity class of $q$, we can assume that the matrix representative $\iota(q)$ for $q$ lies in $\mathbf{GL}_d(\Cal{O}_k)$.

An admissible hyperbolic pair gives rise to a commensurability class of $m$--dimensional hyperbolic orbifolds as follows. Restriction of scalars followed by the appropriate projection induces a map $\pi\colon\mathbf{G}(k)\to \mathbf{PSO}_0(m,1)$ with finite kernel, and we call the image $\Gamma_q = \pi(\mathbf{G}(\Cal{O}_k))$ a \textbf{principle arithmetic lattice} in $\mathbf{PSO}_0(m,1)$. As $\mathbf{PSO}_0(m,1) = \mathrm{Isom}^+(\mathbb{H}^m)$, the lattice $\Gamma_q$ is also the orbifold fundamental group of the orientable hyperbolic orbifold $M_{\Gamma_q}=\Gamma_q \backslash \mathbb{H}^m$.

We call hyperbolic manifolds commensurable with $M_{\Gamma_q}$ \textbf{standard arithmetic manifolds}, and emphasize that every even-dimensional arithmetic hyperbolic manifold is standard. However, when $m$ is odd, there are infinitely many commensurability classes of non-standard arithmetic lattices. See \cite{Meyer} for more details on parametrizing commensurability classes of arithmetic hyperbolic orbifolds.

For any lattice $\Gamma$ in $\mathbf{PSO}_0(m, 1)$, let $\wt{\Gamma}$ be the lift of $\Gamma$ to $\mathbf{SO}_0(m, 1)$. When $m$ is even, the groups $\mathbf{PSO}_0(m, 1)$, $\mathbf{SO}_0(m, 1)$ are isomorphic and so $\wt{\Gamma} \cong \Gamma$. When $m$ is odd, $\mathbf{SO}_0(m, 1)$ is a two-fold covering of $\mathbf{PSO}_0(m, 1)$, and hence we have a central exact sequence
\[
\xymatrix{1 \ar[r] & \mu_2 \ar[r] & \wt{\Gamma} \ar[r] & \Gamma \ar[r] & 1,}
\]
where $\mu_2$, the group of $2^{nd}$ roots of unity, is the center of $\mathbf{SO}_0(m, 1)$. If this exact sequence does not split, there is an index two subgroup of $\Gamma$ for which the associated sequence does split. In other words, possibly passing to an index two subgroup when $m$ is odd, we can assume that $\Gamma$ embeds as a lattice in $\mathbf{SO}_0(m, 1)$.

Associated with any totally geodesic embedding $f\colon \mathbb{H}^n \hookrightarrow \mathbb{H}^m$ is an injection
\[
f_* \colon \mathbf{PS}_0\left(\mathbf{O}(n, 1) \times \mathbf{O}(m - n)\right) \hookrightarrow \mathbf{PSO}_0(m, 1),
\]
and we will denote the image by $H_f$. Given a torsion-free lattice $\Gamma$ in $\mathbf{PO}_0(m, 1)$, proper, totally geodesic, finite volume submanifolds of $M_\Gamma = \Gamma \backslash \mathbb{H}^m$ are then associated with embeddings $f$ as above such that $\Gamma \cap H_f$ is a lattice in $H_f$. Notice that, while $M_\Gamma$ is an orientable manifold, a geodesic submanifold can be nonorientable. Moreover, the submanifold is oriented if and only if $\left(\Gamma \cap H_f \right) \subset f_*\left(\mathbf{P}_0\left(\mathbf{SO}(n, 1) \times \mathbf{SO}(m - n)\right) \right)$.

We now relate $\Gamma \cap H_f$ to the fundamental group of the geodesic submanifold. Let $N_\Lambda = \Lambda \backslash \mathbb{H}^n$ be an oriented totally geodesic submanifold of $M_\Gamma$ of dimension $n$. Then we have an injective homomorphism $\Lambda \to \Gamma$. Choosing a lifting of $N_\Lambda \to M_\Gamma$ to an embedding $f\colon \mathbb{H}^n \hookrightarrow \mathbb{H}^m$ of universal coverings, we see that $\Lambda$ is a subgroup of $\Gamma \cap H_f$. Assuming that $\Gamma$ lifts to $\mathbf{SO}_0(m, 1)$, we obtain in injective homomorphism $f_\star\colon \Lambda \to \mathbf{SO}(n, 1) \times \mathbf{SO}(m-n)$. The real Zariski closure of $f_\star(\Lambda)$ is then of the form $\mathbf{SO}_0(n, 1) \times H_\Lambda$ for some closed subgroup $H_\Lambda$ of $\mathbf{SO}(m-n)$.

As is well-known, an orientable finite volume totally geodesic subspace $N_\Lambda$ of $M_\Gamma$ is also arithmetic \cite[\S3]{Meyer}. Associated with $N_\Lambda$ is an $(n+1)$--dimensional quadratic subform $r$ of $q$ with orthogonal complement $t$ (i.e., $q$ is $k$--isometric to $r\oplus t$) such that the $k$--groups $\mathbf{H}_r=\mathbf{SO}(r)$, $\mathbf{H}_t=\mathbf{SO}(t)$, and $\mathbf{H}=\mathbf{H}_r\times \mathbf{H}_t$ satisfy 
\[
\mathbf{H}_r(\R)=\mathbf{SO}(n,1), \quad \mathbf{H}_t(\R)=\mathbf{SO}(m-n), \quad \mbox{and} \quad \wt{\Lambda}=f_\star(\Lambda)\subset \mathbf{H}(k).
\]
The semisimple $k$--group $\mathbf{H}$ is naturally a $k$--subgroup of $\mathbf{G}$. We call $\Lambda$ \textbf{a totally geodesic subgroup} of either $\Gamma$ or the lift $\wt{\Gamma}$ of $\Gamma$ to $\mathbf{G}(\Cal{O}_k)$; recall from above that $\Lambda$ is isomorphic to a subgroup of both $\Gamma$ and $\wt{\Gamma}$.

\subsection{Strategy of Proof: Geometrically equivalent}

We will find a finite group $G$, a surjective homomorphism $\rho\colon\Gamma\to G$, and two subgroups $C_1, C_2\subset G$ such that
\begin{equation}\label{SubgroupCondition}
\rho(\Lambda)\cap C_1=\rho(\Lambda)\cap C_2
\end{equation}
for all totally geodesic $\Lambda\subset \Gamma$. It then follows from covering space theory that the finite covers $M_1,M_2$ associated with $\Gamma_1=\rho^{-1}(C_1), \Gamma_2=\rho^{-1}(C_2)$ contain exactly the same totally geodesic submanifolds (see \cite[Lemma 4.1]{McR}). Thus, it suffices to find a map $\rho\colon\mathbf{G}(\Cal{O}_k)\to G$ such that $\mathrm{gcd}([\mathbf{G}(\mathcal{O}_K): \ker \pi \cap \mathbf{G}(\Cal{O}_k)],\abs{C_i})=1$ and \eqref{SubgroupCondition} holds. Let $S_0$ denote the set of nondyadic primes of $\Cal{O}_k$ not lying over a prime dividing the index $[\mathbf{G}(\mathcal{O}_K): \ker \pi \cap \mathbf{G}(\Cal{O}_k)]$. The candidates for $G$ and $\rho$ are the natural reduction maps $\rho_{\Fr{p}}\colon \mathbf{G}(\Cal{O}_k) \to \mathbf{G}(\Cal{O}_k/\Fr{p})$, 
where $\mathfrak{p}$ is a prime ideal of $\Cal{O}_k$. Set $\BB{F}_{p^r}=\Cal{O}_k/\Fr{p}$, where $|\Cal{O}_k/\Fr{p}|=p^r$. For a totally geodesic subgroup $\Lambda$, set $H_{\Fr{p}}=\rho_{\Fr{p}}(\wt{\Lambda})$, which sits inside of $\rho_{\Fr{p}}(\mathbf{G}(\Cal{O}_k))$. For our examples, $C_1$ will be the trivial subgroup and $C_\ell$ will be a cyclic group of prime order $\ell$ such that $\ell$ does not divide the order of $H_\Fr{p}$ for any totally geodesic subgroup. In that case, \eqref{SubgroupCondition} will be satisfied and the manifolds $M_1$ and $M_\ell$ associated with the pullbacks of $C_1$ and $C_\ell$ will be geometrically equivalent. Furthermore, notice that, since our covering has odd degree, nonorientable manifolds only lift to nonorientable manifolds, so $\mathrm{TG}(M_1)$, which only contains oriented submanifolds, indeed equals $\mathrm{TG}(M_\ell)$.

Finding the desired prime $\ell$ requires two main steps:

\begin{enumerate}
\item Compute $|\rho_{\Fr{p}}(\mathbf{G}(\Cal{O}_k))|$. This step uses structure theory of algebraic groups, basic Galois cohomology, and strong approximation. We obtain the diagram
\begin{equation}\label{E:Dia1}
\xymatrix{
	&& \widetilde{\mathbf{G}}(\Cal{O}_k)\ar[r] \ar@{->>}[d]_{\rho_p}	&\mathbf{G}(\Cal{O}_k)\ar[d]_{\rho_p}&& \\
1\ar[r]	&\mathbf{F}(\BB{F}_{p^r})\ar[r]& \widetilde{\mathbf{G}}(\BB{F}_{p^r})\ar[r] 	&\mathbf{G}(\BB{F}_{p^r})\ar[r]&H^1(\BB{F}_{p^r},\mathbf{F})\ar[r]&1, 
}
\end{equation}
where $\widetilde{\mathbf{G}}$ is the simply connected cover of $\mathbf{G}$ and $\mathbf{F}$ is a finite $\mathbb{F}_{p^r}$--group.

\item Determine all possible divisors of $|H_\Fr{p}|$. This step uses Bruhat--Tits theoretic computations associated with the diagram
\begin{equation}\label{E:Dia2}
\xymatrix{
\wt{\Lambda}\ar@{^{(}->}[r] \ar[d]_{\rho_p}	&\mathbf{H}(k_{\Fr{p}})\cap\mathbf{G}(\Cal{O}_{\Fr{p}})\ar@{->>}[d] \\
H_{\Fr{p}} 	\ar@{^{(}->}[r]				&\overline{\mathbf{H}}(\mathbb{F}_{p^r}),}
\end{equation}
where $\overline{\mathbf{H}}$ is a certain algebraic $\F_{p^r}$--group associated with $\mathbf{H}$. We know the right vertical map is surjective, and hence we can realize $H_{\Fr{p}}$ as a subgroup of $\overline{\mathbf{H}}(\mathbb{F}_{p^r})$. Recall that $k_{\Fr{p}}$ denotes the localization of $k$ at $\Fr{p}$ and $\Cal{O}_{\Fr{p}}$ is its ring of integers.
\end{enumerate}

Using the calculations for the orders of the groups $\rho_{\Fr{p}}(\mathbf{G}(\Cal{O}_k))$ and the subgroups $H_\Fr{p}$, we find the prime $\ell$ using Zsigmondy's Theorem \cite{Zsig}. 

\subsection{Strategy of Proof: Geometrically isospectral}

Following \cite{McR}, to produce geometrically isospectral manifolds we require two good primes $\Fr{p}_1,\Fr{p}_2$ where we can use the same prime $\ell$ for both $\Fr{p}_1$ and $\Fr{p}_2$ in the above construction. The key observation in using the two primes $\Fr{p}_1,\Fr{p}_2$ is that, since $M_1$ is a cyclic cover of degree $\ell$ to which every geodesic submanifold of $M_\ell$ has exactly $\ell$ distinct lifts, the geometric spectrum of the orbifolds satisfies
\begin{equation}\label{SpectrumRelationship}
\mathcal{TG}(M_1) = \set{(X,m_{X,1})} = \set{(X,\ell m_{X,\ell})},
\end{equation}
where $\mathcal{TG}(M_\ell) = \set{(X,m_{X,\ell})}$. The validity of \eqref{SpectrumRelationship} follows from the argument used in \cite[p.~178]{McR} to establish this for totally geodesic subsurfaces of a hyperbolic 3--manifold. That there exists a prime $\ell$ that satisfies the necessary properties for both $\Fr{p}_1$ and $\Fr{p}_2$ is a straightforward application of the Cebotarev Density Theorem. In particular, there is a positive density set of primes $\Fr{p}$ for which our methods apply.

\section{Step 1: Computing $\abs{\rho_{\Fr{p}}(\mathbf{G}(\Cal{O}_k))}$.}\label{section:orders}

For each $\mathfrak{p} \in S_0$, let $q_{\Fr{p}}$ denote the reduction of $q$ to $\mathcal{O}_k/\mathfrak{q} = \mathbb{F}_{p^r}$. We will say $q$ has a \textbf{good reduction} at $\Fr{p}$ if $q_{\Fr{p}}$ is nondegenerate and note that the subset $S_1 \su S_0$ where $q$ has good reduction is cofinite. For $\mathfrak{p} \in S_1$, set $G_{\mathfrak{p}}=SO(m+1; p^r)$ to be $\mathbb{F}_{p^r}$--points of $\mathbf{SO}(q_{\Fr{p}})$. Over a finite field, orthogonal groups are always quasi-split, and hence come in one of three types (see \cite[Table 1]{Ono} for the orders of these groups):

\begin{enumerate}
\item $B_{n,n}$, the only form of $B_n$, arises when $\dim q=2n+1$. It has order 
\begin{equation}\label{E:Bn}
\abs{SO(2n+1;p^r)} = p^{rn^2}\prod_{j=1}^{n}(p^{2rj}-1).
\end{equation}
\item $D_{n,n}$, the split form of $D_n$, arises when $\dim q=2n$ and $\mathrm{disc}\, q$ is a square in $\mathbb{F}_{p^r}$. It has order
\begin{equation}\label{E:DnSquare}
\abs{SO^+(2n; p^r)} = p^{rn(n-1)}(p^{rn}- 1)\prod_{j=1}^{n-1}(p^{2rj}-1).
\end{equation}
\item $D_{n,n-1}$, the nonsplit quasi-split form of $D_n$, arises when $\dim q=2n$ and $\mathrm{disc}\, q$ is not square in $\mathbb{F}_{p^r}$. It has order
\begin{equation}\label{E:DnNotSquare}
\abs{SO^-(2n; p^r)} = p^{rn(n-1)}(p^{rn}+1)\prod_{j=1}^{n-1}(p^{2rj}-1).
\end{equation}
\end{enumerate}

We have the exact sequence of algebraic $k$--groups (see \cite[\S2.3]{PR})
\[
1\lra \mathbf{\mu}_2\lra \mathbf{Spin}(q)\lra \mathbf{SO}(q)\lra 1,
\]
where $\mu_2$ is the cyclic group of order two. This sequence yields the exact sequence for $\mathbb{F}_{p^r}$--points
\[
1\lra \mu_2\lra \mathbf{Spin}(q)(\mathbb{F}_{p^r})\lra \mathbf{SO}(q)(\mathbb{F}_{p^r})\lra \mathbb{F}_{p^r}^\times / (\mathbb{F}_{p^r}^\times)^2 \lra 1.
\]
Strong approximation (see Lem.~1.1 and Thm.~2.3 in \cite{Ra}) gives us that $\rho_\mathfrak{p}\colon \mathbf{Spin}(q)(\Cal{O}_k)\to \mathbf{Spin}(q)(\mathbb{F}_{p^r})$ is surjective, and we obtain the commutative diagram:
\[
\xymatrix{
	&& \mathbf{Spin}(q)(\Cal{O}_k)\ar[r] \ar@{->>}[d]_{\rho_{\Fr{p}}}	& \mathbf{SO}(q)(\Cal{O}_k)\ar[d]^{\rho_{\Fr{p}}}&& \\
1\ar[r]	&\mu_2\ar[r]&\mathbf{Spin}(q)(\mathbb{F}_{p^r})\ar[r] 	&\mathbf{SO}(q)(\mathbb{F}_{p^r})\ar[r]&\mathbb{F}_{p^r}^\times / (\mathbb{F}_{p^r}^\times)^2\ar[r]&1 }
\]

Using this commutative diagram and noting that $|\mathbb{F}_{p^r}^\times / (\mathbb{F}_{p^r}^\times)^2|=2$, we obtain the following.

\begin{prop}\label{P:Bob}
The index $[G_{\mathfrak{p}} : \rho_{\Fr{p}}(\mathbf{G}(\Cal{O}_k))]$ is either one or two.
\end{prop}

Proposition \ref{P:Bob} with the above list of group orders completes our calculation of $\abs{\rho_{\Fr{p}}(\mathbf{G}(\Cal{O}_k))}$.

\section{Step 2: Computing $\abs{H_{\Fr{p}}}$ for a totally geodesic $\wt{\Lambda}$.}\label{sec:Step2}

Our goal of this section is the computations of $\abs{H_{\Fr{p}}}$ for a generic totally geodesic $\wt{\Lambda}\subset\B{G}(\Cal{O}_k)$. We use the notation established in \S \ref{section:overview}. Let $\Fr{p}\in S_1$ and $\Cal{G}_{\Fr{p}}=\B{G}(\Cal{O}_{\Fr{p}})$ denote the parahoric of $\B{G}(k_{\Fr{p}})$ with pro--$p$ unipotent radical $\Cal{G}_{\Fr{p}}^+$. It follows that $\Cal{H}_{\Fr{p}}=\B{H}(k_{\Fr{p}})\cap \Cal{G}_{\Fr{p}}$ is a parahoric of $\B{H}(k_{\Fr{p}})$ containing $\wt{\Lambda}$, and $\Cal{H}_{\Fr{p}}^+=\Cal{G}_{\Fr{p}}^+\cap \Cal{H}_{\Fr{p}}$ is the pro--$p$ unipotent radical of $\Cal{H}_{\Fr{p}}$. Set $\overline{\mathbf{H}}$ to be the $\mathbb{F}_{p^r}$--group whose $\mathbb{F}_{p^r}$--points are $\Cal{H}_{\Fr{p}}/\Cal{H}_{\Fr{p}}^+$. We have the following commutative diagram where we know the right two vertical arrows are surjections by \cite[3.4.4]{T2}.\\
\centerline{\hfill
\xymatrix{
\wt{\Lambda}\ar@{^{(}->}[r] \ar[d]_{\rho_{\Fr{p}}}	&\Cal{H}_{\Fr{p}}\ar@{^{(}->}[r]\ar@{->>}[d] \ar[d]	&\Cal{G}_{\Fr{p}}\ar@{->>}[d] \\
H_{\Fr{p}} 	\ar@{^{(}->}[r]				&\overline{\mathbf{H}}(\mathbb{F}_{p^r})	\ar@{^{(}->}[r]				&SO(m+1, p^r) }\hfill}
It follows that $H_{\Fr{p}}$ is a subgroup of $\overline{\mathbf{H}}(\mathbb{F}_{p^r})$, which is in turn a subgroup of $SO(m+1, p^r)$. 

\subsection{A simplification} 

The group $\overline{\mathbf{H}}(\mathbb{F}_{p^r})$ fits into the following exact sequence
\begin{equation}\label{E:Seq1}
1 \lra \mathcal{R}_u(\overline{\mathbf{H}})(\mathbb{F}_{p^r}) \lra \overline{\mathbf{H}}(\mathbb{F}_{p^r})\lra \overline{\mathbf{H}}^{red}(\mathbb{F}_{p^r})\lra 1,
\end{equation}
where $\overline{\mathbf{H}}^{red}$ is a reductive group whose Dynkin diagram can be read off of local Dynkin diagrams. From \eqref{E:Seq1} we obtain
\begin{align}
\abs{\overline{\mathbf{H}}(\mathbb{F}_{p^r})} = \abs{\mathcal{R}_u(\overline{\mathbf{H}})(\mathbb{F}_{p^r})}\cdot\abs{\overline{\mathbf{H}}^{red}(\mathbb{F}_{p^r})}.
\end{align}
Therefore, computing $|\overline{\mathbf{H}}(\mathbb{F}_{p^r})|$ reduces to computing the size of unipotent $\mathbb{F}_{p^r}$--groups and the size of $\overline{\mathbf{H}}^{red}(\mathbb{F}_{p^r})$. We compute the former with the following proposition.

\begin{prop}\label{unip}
If $\mathbf{U}$ is a unipotent group over a finite field $\mathbb{F}_{p^r}$, then $|\mathbf{U}(\mathbb{F}_{p^r})|=p^s$ for some $s\in \Z_{\ge0}$. 
\end{prop}

\begin{proof}
Since $\mathbb{F}_{p^r}$ is perfect, $\mathbf{U}$ splits \cite[15.5(ii)]{B}. Therefore $\mathbf{U}$ admits a composition series
\[
\mathbf{U}=\mathbf{U}_0\supset \mathbf{U}_1 \supset \mathbf{U}_2\cdots \supset \mathbf{U}_s = \{1\}
\] 
of connected $\mathbb{F}_{p^r}$--groups such that $\mathbf{U}_i/\mathbf{U}_{i+1}$ is $\mathbb{F}_{p^r}$--isomorphic to $\mathbf{G}_a$. Since each $\mathbf{U}_{i+1}$ is connected, by Lang's theorem \cite[6.1]{PR} $H^1(\mathbb{F}_{p^r},\mathbf{U}_{i+1})$ is trivial, and hence 
\[
1 \lra \mathbf{U}_{i+1}(\mathbb{F}_{p^r}) \lra \mathbf{U}_i(\mathbb{F}_{p^r}) \lra \mathbf{G}_a(\mathbb{F}_{p^r}) \lra 1
\] 
is exact. We proceed by inducting on the length of the composition series. If the series has length $0$, then $\mathbf{U}\cong \mathbf{G}_a$, and hence $|\mathbf{U}(\mathbb{F}_{p^r})|=p^r$. If the statement is true for series of length $j$, then the above exact sequence implies it follows for series of length $j+1$, and the result follows.
\end{proof}

\subsection{Computing $|\overline{\mathbf{H}}^{red}(\mathbb{F}_{p^r})|$.}

We are now left computing the orders of $\overline{\mathbf{H}}^{red}(\mathbb{F}_{p^r})$. To do so, we use the classification of local indices \cite{T2}. A $\Fr{p}$--adic group $\mathbf{H}$ is called \textbf{residually split} if $\mathrm{rank}_{k_{\Fr{p}}}(\mathbf{H})=\mathrm{rank}_{k_{\Fr{p}}^{un}}(\mathbf{H})$ where $k_{\Fr{p}}^{un}$ is the maximal unramified extension of $k_{\Fr{p}}$. The classification of local Dynkin diagrams of simple $k_{\Fr{p}}$--groups falls into two classes, residually split and not residually split. As we explain later, we can restrict ourselves to computing these orders for totally geodesic groups of ``maximal dimension'' for both $\overline{\mathbf{H}}_r^{red}(\mathbb{F}_{p^r})$ and $\overline{\mathbf{H}}_t^{red}(\mathbb{F}_{p^r})$.

\begin{prop}\label{comp1}
Continuing the notation of the earlier sections, suppose $\mathbf{H}_0=\mathbf{SO}(q_0)$ for some quadratic subform $q_0\subset q$ of odd dimension $2n-1\ge 4$ and let $\Fr{p}\subset S_1$. Then $|\overline{\mathbf{H}}_0^{red}(\mathbb{F}_{p^r})|$ divides $p^XY$ where $X\in \Z_{\ge 0}$ and $Y$ is one of the following:
\begin{enumerate}[\qquad (T1)]
\item $ \prod_{j=1}^{n-1}(p^{2rj}-1) $,
\item $ (p^{2r}-1)^2\ \prod_{j=1}^{n-3}(p^{2rj}-1)$,
\item $(p^{r(k-1)}\pm 1)\left(\prod_{j=1}^{k-2}(p^{2rj}-1)\right) \left(\prod_{j=1}^{n-k}(p^{2rj}-1)\right) $ for $3\le k\le n-3$,
\item $ (p^{2r}-1) (p^{r(n-2)}\pm 1)\prod_{j=1}^{n-3}(p^{2rj}-1)$,
\item $ (p^{r(n-1)}\pm 1)\prod_{j=1}^{n-2}(p^{2rj}-1)$,
\item $\prod_{j=1}^{n-2}(p^{2rj}-1)$,
\item $(p^{2r}-1)\ \prod_{j=1}^{n-3}(p^{2rj}-1)$,
\item $\left(\prod_{j=1}^{k-1}(p^{2rj}-1)\right)\ \left(\prod_{j=1}^{n-k-1}(p^{2rj}-1)\right) $ for $3\le k\le n-3$.
\end{enumerate}
\end{prop}

\begin{proof}${}$
Since every parahoric lies in a maximal one it suffices to compute the orders of all possible reductions of maximal parahorics. We analyze all possible local indices of $\mathbf{H}$ and remove one vertex to obtain the Dynkin diagram of $\overline{\mathbf{H}}^{red}$ \cite{T2}. We then use the orders of Section \ref{section:orders}, \cite{Ono}, and Proposition \ref{unip} to compute the size of each possible quotient. For each case below, we give the local diagram, where we have distinguished the nodes associated with similar reductions. We follow the diagram with a table listing the Killing--Cartan type and order of the reduction group associated with each class of node.

Case 1 - $\mathbf{H}$ is residually split of type $B_{n-1}$.

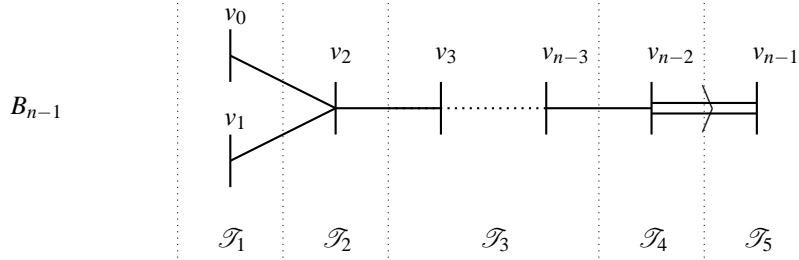
\begin{center}
  \begin{tikzpicture}[scale=.7]
    \draw (-3,0) node[anchor=east]  {$B_{n-1}$};
    \foreach \x in {1,...,5}
    \draw[xshift=\x cm,thick] (\x cm,-0.5) -- (\x, 0.5 cm);
    \foreach \x in {1,3}
    \draw[xshift=\x cm,thick] (\x cm,0) -- +(2 cm, 0cm);
   \draw[dotted,thick] (3 cm,0) -- +(3 cm, 0cm);
    \draw[thick] (0 cm, 0.5 cm) -- +(0 cm,1.0cm);
    \draw[thick] (0 cm, -0.5 cm) -- +(0 cm,-1.0cm);
    \draw[thick] (0 cm, -1 cm) -- (2 cm,0cm);
    \draw[thick] (0 cm, 1 cm) -- (2 cm,0cm);
    \draw[thick] (8 cm, .1 cm) -- +(2 cm,0);
    \draw[thick] (8 cm, -.1 cm) -- +(2 cm,0);
    \draw (9.4,0) node[anchor=east]  {$\Big>$};
    
    \draw[dotted, thin] (-1 cm, 2 cm) -- (-1 cm,-3cm);
    \draw[dotted, thin] (1 cm, 2 cm) -- (1 cm,-3cm);
    \draw[dotted, thin] (3 cm, 2 cm) -- (3 cm,-3cm);
    \draw[dotted, thin] (7 cm, 2 cm) -- (7 cm,-3cm);
    \draw[dotted, thin] (9 cm, 2 cm) -- (9 cm,-3cm);
    \draw[dotted, thin] (11 cm, 2 cm) -- (11 cm,-3cm);
    \draw (0.5cm,-2.5cm) node[anchor=east]  {$\mathcal{T}_1$};
    \draw (2.5cm,-2.5cm) node[anchor=east]  {$\mathcal{T}_2$};
    \draw (5.5cm,-2.5cm) node[anchor=east]  {$\mathcal{T}_3$};
    \draw (8.5cm,-2.5cm) node[anchor=east]  {$\mathcal{T}_4$};
    \draw (10.5cm,-2.5cm) node[anchor=east]  {$\mathcal{T}_5$};
    
    \draw (0.5cm,1.75cm) node[anchor=east]  {$v_0$};
    \draw (0.5cm,-0.25cm) node[anchor=east]  {$v_1$};
    \draw (2.5cm,1cm) node[anchor=east]  {$v_2$};
    \draw (4.5cm,1cm) node[anchor=east]  {$v_3$};
    \draw (7cm,1cm) node[anchor=east]  {$v_{n-3}$};
    \draw (9cm,1cm) node[anchor=east]  {$v_{n-2}$};
    \draw (11cm,1cm) node[anchor=east]  {$v_{n-1}$};
  \end{tikzpicture}
\end{center}

 \begin{adjustwidth}{-0.5cm}{}
\begin{center}
    \begin{tabular}{ | c | l | l |}
    \hline
    & Type of $\overline{\mathbf{H}}^{red}$ & Order of $\overline{\mathbf{H}}^{red}$ \\ \hline
    $\mathcal{T}_1$ & $B_{n-1}$ & $p^{r(n-1)^2}\prod_{j=1}^{n-1}(p^{2rj}-1)$ \\ \hline
    $\mathcal{T}_2$ & $A_1\times A_1 \times B_{n-3}$ 
    & $(p^r(p^{2r}-1))^2\ \bigg(p^{r(n-3)^2}\prod_{j=1}^{n-3}(p^{2rj}-1)\bigg)$ \\ \hline
    $\mathcal{T}_3$ & $D_{k}\times B_{n-k-1}$\ ($3 \le k\le n-3$)
    & $\bigg(p^{rk(k-1)}(p^{rk}\pm 1)\prod_{j=1}^{k-1}(p^{2rj}-1)\bigg) \ \bigg(p^{r(n-k-1)^2}\prod_{j=1}^{n-k-1}(p^{2rj}-1)\bigg)$ \\ \hline
    $\mathcal{T}_4$ & $D_{n-2}\times A_1$& $ \bigg(p^{r(n-2)(n-3)}(p^{r(n-2)}\pm 1)\prod_{j=1}^{n-3}(p^{2rj}-1)\bigg) (p^r(p^{2r}-1))$ \\ \hline
    $\mathcal{T}_5$ & $D_{n-1}$ & $p^{r(n-1)(n-2)}(p^{r(n-1)}\pm 1)\prod_{j=1}^{n-2}(p^{2rj}-1)$ \\
    \hline
    \end{tabular}
\end{center}
\end{adjustwidth}

Case 2 - $\mathbf{H}$ is not residually split of type $B_{n-1}$.

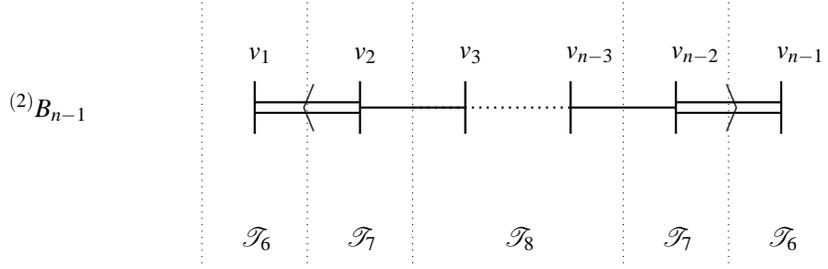
\begin{center}
  \begin{tikzpicture}[scale=.7]
    \draw (-3,0) node[anchor=east]  {${}^{(2)}B_{n-1}$};
    \foreach \x in {0,1,...,5}
    \draw[xshift=\x cm,thick] (\x cm,-0.5) -- (\x, 0.5 cm);
    \foreach \x in {1,3}
    \draw[xshift=\x cm,thick] (\x cm,0) -- +(2 cm, 0cm);
   \draw[dotted,thick] (3 cm,0) -- +(3 cm, 0cm);
    \draw[thick] (8 cm, .1 cm) -- +(2 cm,0);
    \draw[thick] (8 cm, -.1 cm) -- +(2 cm,0);
    \draw[thick] (0 cm, .1 cm) -- +(2 cm,0);
    \draw[thick] (0 cm, -.1 cm) -- +(2 cm,0);
    \draw (9.4,0) node[anchor=east]  {$\Big>$};
    \draw (1.35,0) node[anchor=east]  {$\Big<$};
    
    \draw[dotted, thin] (-1 cm, 2 cm) -- (-1 cm,-3cm);
    \draw[dotted, thin] (1 cm, 2 cm) -- (1 cm,-3cm);
    \draw[dotted, thin] (3 cm, 2 cm) -- (3 cm,-3cm);
    \draw[dotted, thin] (7 cm, 2 cm) -- (7 cm,-3cm);
    \draw[dotted, thin] (9 cm, 2 cm) -- (9 cm,-3cm);
    \draw[dotted, thin] (11 cm, 2 cm) -- (11 cm,-3cm);
    \draw (0.5cm,-2.5cm) node[anchor=east]  {$\mathcal{T}_6$};
    \draw (2.5cm,-2.5cm) node[anchor=east]  {$\mathcal{T}_7$};
    \draw (5.5cm,-2.5cm) node[anchor=east]  {$\mathcal{T}_8$};
    \draw (8.5cm,-2.5cm) node[anchor=east]  {$\mathcal{T}_7$};
    \draw (10.5cm,-2.5cm) node[anchor=east]  {$\mathcal{T}_6$};
    
    \draw (0.5cm,1cm) node[anchor=east]  {$v_1$};
    \draw (2.5cm,1cm) node[anchor=east]  {$v_2$};
    \draw (4.5cm,1cm) node[anchor=east]  {$v_3$};
    \draw (7cm,1cm) node[anchor=east]  {$v_{n-3}$};
    \draw (9cm,1cm) node[anchor=east]  {$v_{n-2}$};
    \draw (11cm,1cm) node[anchor=east]  {$v_{n-1}$};
  \end{tikzpicture}
\end{center}

\begin{center}
    \begin{tabular}{ | c | l | l |}
    \hline
    & Type of $\overline{\mathbf{H}}^{red}$ & Order of $\overline{\mathbf{H}}^{red}$ \\ \hline
    $\mathcal{T}_6$ & $B_{n-2}$ & $p^{r(n-2)^2}\prod_{j=1}^{n-2}(p^{2rj}-1)$ \\ \hline
    $\mathcal{T}_7$ & $A_1 \times B_{n-3}$ 
    & $(p^r(p^{2r}-1))\ \bigg(p^{r(n-3)^2}\prod_{j=1}^{n-3}(p^{2rj}-1)\bigg)$ \\ \hline
    $\mathcal{T}_8$ & $B_{k-1}\times B_{n-k-1}$\ ($3\le k\le n-3$)
    & $\bigg(p^{r(k-1)^2}\prod_{j=1}^{k-1}(p^{2rj}-1)\bigg) \ \bigg(p^{r(n-k-1)^2}\prod_{j=1}^{n-k-1}(p^{2rj}-1)\bigg)$ \\ \hline
     \end{tabular}
\end{center}
\leftskip0pc
\end{proof}

\begin{prop}\label{comp2}
Continuing the notation of the earlier sections, suppose $\mathbf{H}_0=\mathbf{SO}(q_0)$ for some quadratic subform $q_0\subset q$ of even dimension $2n\ge 4$ and let $\Fr{p}\subset S_1$. 
Then $\left|\overline{\mathbf{H}}_0^{red}(\mathbb{F}_{p^r})\right|$ divides $p^XY$ where $X\in \Z_{\ge 0}$ and $Y$ is one of the following:
\begin{enumerate}[\qquad (S1)]
\item $(p^{rn}\pm 1)\prod_{j=1}^{n-1}(p^{2rj}-1)$,
\item $(p^{2r}-1)^2\ (p^{r(n-2)}\pm 1)\prod_{j=1}^{n-3}(p^{2rj}-1)$,
\item $(p^{rk}\pm 1)(p^{r(n-k)}\pm 1) \Big(\prod_{j=1}^{k-1}(p^{2rj}-1)\Big) \ \Big(\prod_{j=1}^{n-k-1}(p^{2rj}-1)\Big)$ for $3 \le k\le n-3$,
\item $\prod_{j=1}^{n-1}(p^{2rj}-1)$,
\item $(p^{2r}-1)\ \prod_{j=1}^{n-2}(p^{2rj}-1)$,
\item $\Big(\prod_{j=1}^{k-1}(p^{2rj}-1)\Big)\ \Big(\prod_{j=1}^{n-k}(p^{2rj}-1)\Big)$ for $3 \le k\le n-2$,
\end{enumerate}
or any of (T1) through (T8) listed in the previous proposition.
\end{prop}

\begin{proof}${}$
The idea and presentation of this proof is the same as for Proposition \ref{comp1}.
		
Case 1 - $\mathbf{H}$ is residually split of type $D_{n}$ and in fact $\mathbf{H}$ splits over $k_{\Fr{p}}$.

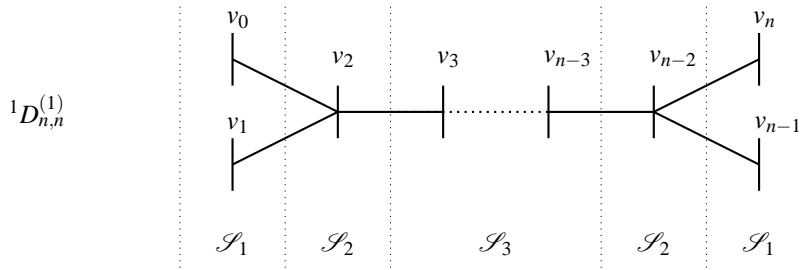
\begin{center}
  \begin{tikzpicture}[scale=.7]
    \draw (-3,0) node[anchor=east]  {${}^1D^{(1)}_{n,n}$};
    \foreach \x in {1,...,4}
    \draw[xshift=\x cm,thick] (\x cm,-0.5) -- (\x, 0.5 cm);
    \foreach \x in {1,3}
    \draw[xshift=\x cm,thick] (\x cm,0) -- +(2 cm, 0cm);
   \draw[dotted,thick] (3 cm,0) -- +(3 cm, 0cm);
    \draw[thick] (0 cm, 0.5 cm) -- +(0 cm,1.0cm);
    \draw[thick] (0 cm, -0.5 cm) -- +(0 cm,-1.0cm);
    \draw[thick] (0 cm, -1 cm) -- (2 cm,0cm);
    \draw[thick] (0 cm, 1 cm) -- (2 cm,0cm);
    \draw[thick] (8 cm, .0 cm) -- +(2 cm,1cm);
    \draw[thick] (8 cm, 0 cm) -- +(2 cm,-1cm);
    \draw[thick] (10 cm, 0.5 cm) -- +(0 cm,1.0cm);
    \draw[thick] (10 cm, -0.5 cm) -- +(0 cm,-1.0cm);
    
    \draw[dotted, thin] (-1 cm, 2 cm) -- (-1 cm,-3cm);
    \draw[dotted, thin] (1 cm, 2 cm) -- (1 cm,-3cm);
    \draw[dotted, thin] (3 cm, 2 cm) -- (3 cm,-3cm);
    \draw[dotted, thin] (7 cm, 2 cm) -- (7 cm,-3cm);
    \draw[dotted, thin] (9 cm, 2 cm) -- (9 cm,-3cm);
    \draw[dotted, thin] (11 cm, 2 cm) -- (11 cm,-3cm);
    \draw (0.5cm,-2.5cm) node[anchor=east]  {$\mathcal{S}_1$};
    \draw (2.5cm,-2.5cm) node[anchor=east]  {$\mathcal{S}_2$};
    \draw (5.5cm,-2.5cm) node[anchor=east]  {$\mathcal{S}_3$};
    \draw (8.5cm,-2.5cm) node[anchor=east]  {$\mathcal{S}_2$};
    \draw (10.5cm,-2.5cm) node[anchor=east]  {$\mathcal{S}_1$};
    
    \draw (0.5cm,1.75cm) node[anchor=east]  {$v_0$};
    \draw (0.5cm,-0.25cm) node[anchor=east]  {$v_1$};
    \draw (2.5cm,1cm) node[anchor=east]  {$v_2$};
    \draw (4.5cm,1cm) node[anchor=east]  {$v_3$};
    \draw (7cm,1cm) node[anchor=east]  {$v_{n-3}$};
    \draw (9cm,1cm) node[anchor=east]  {$v_{n-2}$};
    \draw (11cm,-0.25cm) node[anchor=east]  {$v_{n-1}$};
    \draw (10.5cm,1.75cm) node[anchor=east]  {$v_{n}$};
  \end{tikzpicture}
  \end{center}

 \begin{adjustwidth}{-1.0cm}{}
\begin{center}
    \begin{tabular}{ | c | l | l |}
    \hline
    & Type of $\overline{\mathbf{H}}^{red}$ & Order of $\overline{\mathbf{H}}^{red}$ \\ \hline
    $\mathcal{S}_1$ & $D_{n}$ & $p^{rn(n-1)}(p^{rn}\pm 1)\prod_{j=1}^{n-1}(p^{2rj}-1)$ \\ \hline
    $\mathcal{S}_2$ & $A_1\times A_1 \times D_{n-2}$ 
    & $(p^r(p^{2r}-1))^2\ \bigg(p^{r(n-2)(n-3)}(p^{r(n-2)}\pm 1)\prod_{j=1}^{n-3}(p^{2rj}-1)\bigg)$ \\ \hline
    $\mathcal{S}_3$ & $D_{k}\times D_{n-k}$\ ($3 \le k\le n-3$)
    & $\bigg(p^{rk(k-1)}(p^{rk}\pm 1)\prod_{j=1}^{k-1}(p^{2rj}-1)\bigg) \ \bigg(p^{r(n-k)(n-k-1)}(p^{r(n-k)}\pm 1)\prod_{j=1}^{n-k-1}(p^{2rj}-1))\bigg)$ \\
    \hline
    \end{tabular}
\end{center}
\end{adjustwidth}

Case 2 - $\mathbf{H}$ is residually split of type $D_{n}$ where $\mathbf{H}$ is nonsplit quasisplit over both $k_{\Fr{p}}$ and $k_{\Fr{p}}^{un}$.

  
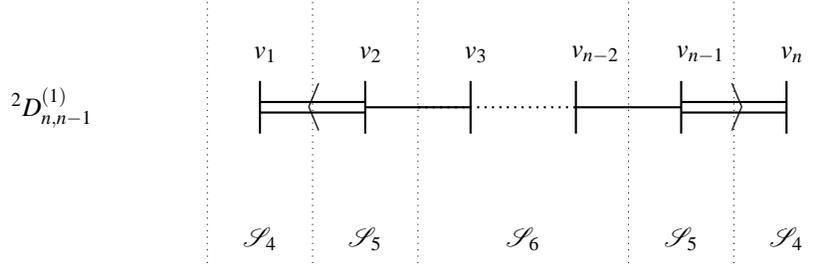
\begin{center}
  \begin{tikzpicture}[scale=.7]
    \draw (-3,0) node[anchor=east]  {${}^{2}D^{(1)}_{n,n-1}$};
    \foreach \x in {0,1,...,5}
    \draw[xshift=\x cm,thick] (\x cm,-0.5) -- (\x, 0.5 cm);
    \foreach \x in {1,3}
    \draw[xshift=\x cm,thick] (\x cm,0) -- +(2 cm, 0cm);
   \draw[dotted,thick] (3 cm,0) -- +(3 cm, 0cm);
    \draw[thick] (8 cm, .1 cm) -- +(2 cm,0);
    \draw[thick] (8 cm, -.1 cm) -- +(2 cm,0);
    \draw[thick] (0 cm, .1 cm) -- +(2 cm,0);
    \draw[thick] (0 cm, -.1 cm) -- +(2 cm,0);
    \draw (9.4,0) node[anchor=east]  {$\Big>$};
    \draw (1.35,0) node[anchor=east]  {$\Big<$};
    
    \draw[dotted, thin] (-1 cm, 2 cm) -- (-1 cm,-3cm);
    \draw[dotted, thin] (1 cm, 2 cm) -- (1 cm,-3cm);
    \draw[dotted, thin] (3 cm, 2 cm) -- (3 cm,-3cm);
    \draw[dotted, thin] (7 cm, 2 cm) -- (7 cm,-3cm);
    \draw[dotted, thin] (9 cm, 2 cm) -- (9 cm,-3cm);
    \draw[dotted, thin] (11 cm, 2 cm) -- (11 cm,-3cm);
    \draw (0.5cm,-2.5cm) node[anchor=east]  {$\mathcal{S}_4$};
    \draw (2.5cm,-2.5cm) node[anchor=east]  {$\mathcal{S}_5$};
    \draw (5.5cm,-2.5cm) node[anchor=east]  {$\mathcal{S}_6$};
    \draw (8.5cm,-2.5cm) node[anchor=east]  {$\mathcal{S}_5$};
    \draw (10.5cm,-2.5cm) node[anchor=east]  {$\mathcal{S}_4$};
    
    \draw (0.5cm,1cm) node[anchor=east]  {$v_1$};
    \draw (2.5cm,1cm) node[anchor=east]  {$v_2$};
    \draw (4.5cm,1cm) node[anchor=east]  {$v_3$};
    \draw (7cm,1cm) node[anchor=east]  {$v_{n-2}$};
    \draw (9cm,1cm) node[anchor=east]  {$v_{n-1}$};
    \draw (10.5cm,1cm) node[anchor=east]  {$v_{n}$};
  \end{tikzpicture}
\end{center}

\begin{center}
    \begin{tabular}{ | c | l | l |}
    \hline
    & Type of $\overline{\mathbf{H}}^{red}$ & Order of $\overline{\mathbf{H}}^{red}$ \\ \hline
    $\mathcal{S}_4$ & $B_{n-1}$ & $p^{r(n-1)^2}\prod_{j=1}^{n-1}(p^{2rj}-1)$ \\ \hline
    $\mathcal{S}_5$ & $A_1\times B_{n-2}$ 
    & $(p^r(p^{2r}-1))\ \bigg(p^{r(n-2)^2}\prod_{j=1}^{n-2}(p^{2rj}-1)\bigg)$ \\ \hline
    $\mathcal{S}_6$ & $B_{k-1}\times B_{n-k}$\ ($3 \le k\le n-2$)
    & $\bigg(p^{r(k-1)^2}\prod_{j=1}^{k-1}(p^{2rj}-1)\bigg) \ \bigg(p^{r(n-k)^2}\prod_{j=1}^{n-k}(p^{2rj}-1)\bigg)$ \\
    \hline
    \end{tabular}
\end{center}

Case 3 - $\mathbf{H}$ is not residually split of type $D_n$ where $\mathbf{H}$ is nonsplit quasi-split over $k_{\Fr{p}}$ but splits over $k_{\Fr{p}}^{un}$.

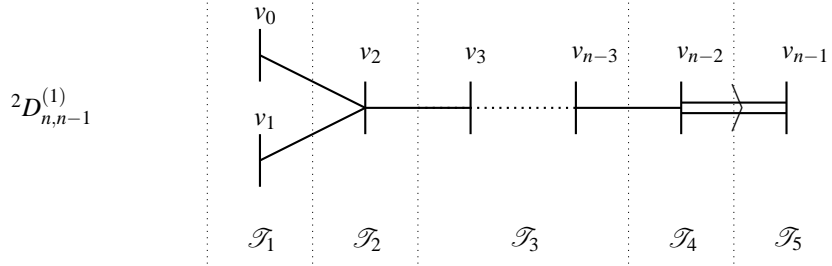
\begin{center}
  \begin{tikzpicture}[scale=.7]
    \draw (-3,0) node[anchor=east]  {${}^2D^{(1)}_{n,n-1}$};
    \foreach \x in {1,...,5}
    \draw[xshift=\x cm,thick] (\x cm,-0.5) -- (\x, 0.5 cm);
    \foreach \x in {1,3}
    \draw[xshift=\x cm,thick] (\x cm,0) -- +(2 cm, 0cm);
   \draw[dotted,thick] (3 cm,0) -- +(3 cm, 0cm);
    \draw[thick] (0 cm, 0.5 cm) -- +(0 cm,1.0cm);
    \draw[thick] (0 cm, -0.5 cm) -- +(0 cm,-1.0cm);
    \draw[thick] (0 cm, -1 cm) -- (2 cm,0cm);
    \draw[thick] (0 cm, 1 cm) -- (2 cm,0cm);
    \draw[thick] (8 cm, .1 cm) -- +(2 cm,0);
    \draw[thick] (8 cm, -.1 cm) -- +(2 cm,0);
    \draw (9.4,0) node[anchor=east]  {$\Big>$};
    
    \draw[dotted, thin] (-1 cm, 2 cm) -- (-1 cm,-3cm);
    \draw[dotted, thin] (1 cm, 2 cm) -- (1 cm,-3cm);
    \draw[dotted, thin] (3 cm, 2 cm) -- (3 cm,-3cm);
    \draw[dotted, thin] (7 cm, 2 cm) -- (7 cm,-3cm);
    \draw[dotted, thin] (9 cm, 2 cm) -- (9 cm,-3cm);
    \draw[dotted, thin] (11 cm, 2 cm) -- (11 cm,-3cm);
    \draw (0.5cm,-2.5cm) node[anchor=east]  {$\mathcal{T}_{1}$};
    \draw (2.5cm,-2.5cm) node[anchor=east]  {$\mathcal{T}_{2}$};
    \draw (5.5cm,-2.5cm) node[anchor=east]  {$\mathcal{T}_{3}$};
    \draw (8.5cm,-2.5cm) node[anchor=east]  {$\mathcal{T}_{4}$};
    \draw (10.5cm,-2.5cm) node[anchor=east]  {$\mathcal{T}_{5}$};
    
    \draw (0.5cm,1.75cm) node[anchor=east]  {$v_0$};
    \draw (0.5cm,-0.25cm) node[anchor=east]  {$v_1$};
    \draw (2.5cm,1cm) node[anchor=east]  {$v_2$};
    \draw (4.5cm,1cm) node[anchor=east]  {$v_3$};
    \draw (7cm,1cm) node[anchor=east]  {$v_{n-3}$};
    \draw (9cm,1cm) node[anchor=east]  {$v_{n-2}$};
    \draw (11cm,1cm) node[anchor=east]  {$v_{n-1}$};
  \end{tikzpicture}
\end{center}

Case 4 - $\mathbf{H}$ is not residually split of type $D_n$ where $\mathbf{H}$ is not quasi-split over $k_{\Fr{p}}$, but splits over $k_{\Fr{p}}^{un}$.

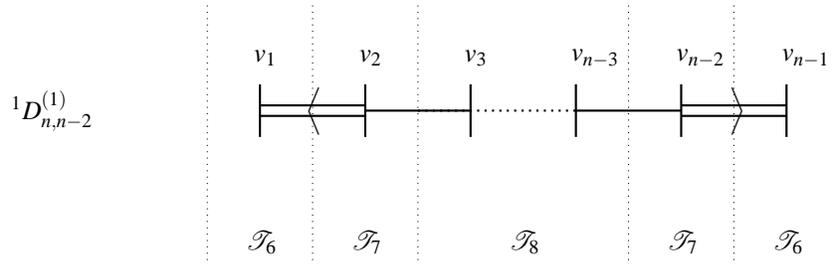
\begin{center}
  \begin{tikzpicture}[scale=.7]
    \draw (-3,0) node[anchor=east]  {${}^{1}D^{(1)}_{n,n-2}$};
    \foreach \x in {0,1,...,5}
    \draw[xshift=\x cm,thick] (\x cm,-0.5) -- (\x, 0.5 cm);
    \foreach \x in {1,3}
    \draw[xshift=\x cm,thick] (\x cm,0) -- +(2 cm, 0cm);
   \draw[dotted,thick] (3 cm,0) -- +(3 cm, 0cm);
    \draw[thick] (8 cm, .1 cm) -- +(2 cm,0);
    \draw[thick] (8 cm, -.1 cm) -- +(2 cm,0);
    \draw[thick] (0 cm, .1 cm) -- +(2 cm,0);
    \draw[thick] (0 cm, -.1 cm) -- +(2 cm,0);
    \draw (9.4,0) node[anchor=east]  {$\Big>$};
    \draw (1.35,0) node[anchor=east]  {$\Big<$};
    
    \draw[dotted, thin] (-1 cm, 2 cm) -- (-1 cm,-3cm);
    \draw[dotted, thin] (1 cm, 2 cm) -- (1 cm,-3cm);
    \draw[dotted, thin] (3 cm, 2 cm) -- (3 cm,-3cm);
    \draw[dotted, thin] (7 cm, 2 cm) -- (7 cm,-3cm);
    \draw[dotted, thin] (9 cm, 2 cm) -- (9 cm,-3cm);
    \draw[dotted, thin] (11 cm, 2 cm) -- (11 cm,-3cm);
    \draw (0.5cm,-2.5cm) node[anchor=east]  {$\mathcal{T}_6$};
    \draw (2.5cm,-2.5cm) node[anchor=east]  {$\mathcal{T}_{7}$};
    \draw (5.5cm,-2.5cm) node[anchor=east]  {$\mathcal{T}_{8}$};
    \draw (8.5cm,-2.5cm) node[anchor=east]  {$\mathcal{T}_{7}$};
    \draw (10.5cm,-2.5cm) node[anchor=east]  {$\mathcal{T}_6$};
    
    \draw (0.5cm,1cm) node[anchor=east]  {$v_1$};
    \draw (2.5cm,1cm) node[anchor=east]  {$v_2$};
    \draw (4.5cm,1cm) node[anchor=east]  {$v_3$};
    \draw (7cm,1cm) node[anchor=east]  {$v_{n-3}$};
    \draw (9cm,1cm) node[anchor=east]  {$v_{n-2}$};
    \draw (11cm,1cm) node[anchor=east]  {$v_{n-1}$};
  \end{tikzpicture}
\end{center}

Observe that these last two diagrams are precisely the the same as the diagrams analyzed in the previous proof, and hence the corresponding Killing--Cartan types and orders are the same.
\end{proof}

\section{Proof of Theorem \ref{thmA}}

Recall that $G_\mathfrak{p}=\mathbf{G}(\mathcal{O}_k/\mathfrak{p})=SO(m+1,p^r)$, and in the previous two sections, we analyzed the orders of its subgroups
 $\rho_{\Fr{p}}(\mathbf{G}(\Cal{O}_k))$ and $H_\mathfrak{p}$.
 To now prove Theorem \ref{thmA}, we need the following result of Zsigmondy \cite{Zsig}.

\begin{thm}[Zsigmondy]\label{ZsigCor}
Let $p$ be an odd prime and $d$ be an integer greater than one. There exists a prime divisor of $p^d+1$ that does not divide $p^j+1$ for all $0 < j < d$ and does not divide $p^j-1$ for all $0 < j < 2d$.
\end{thm}

\begin{lemma}\label{L:PrimeCor}
Let $(k,q)$ be an admissible hyperbolic pair and $S_1$ the set of nondyadic primes in $\mathcal{O}_k$ where $q$ has good reduction. Then for each $\mathfrak{p} \in S_1$, there exists a subgroup $C_\mathfrak{p} < G_\mathfrak{p}$ such that $C_\mathfrak{p} \cap H_\mathfrak{p} = \set{1}$ for any $H_\mathfrak{p}$.
\end{lemma}

\begin{proof}
When $\dim(q) = 2n+1$, we know that $p^{nr}+1$ divides $\abs{G_\mathfrak{p}}$ for any prime $\mathfrak{p} \in S_1$ by \eqref{E:Bn}. For the groups $H_\Fr{p}$, we know that $\abs{H_\Fr{p}}$ divides $p^\alpha\prod_j (p^j-1) \prod_{j'} (p^{j'}+1)$, where $j \leq 2r(n-1)$ and $j' \leq r(n-1)$. Consequently, $p^{nr}+1$ is not a divisor of $\abs{H_\mathfrak{p}}$ for any totally geodesic subgroup. By Theorem \ref{ZsigCor}, there exists a prime divisor $\ell_\mathfrak{p}$ of $p^{nr}+1$ that does not divide $p^j+1$ for $0 < j < nr$ or $p^{2jr}-1$ for $0 < j < n$. It follows that $\ell_\mathfrak{p}$ divides $\abs{G_\Fr{p}}$ but not $\abs{H_\Fr{p}}$ for any totally geodesic subgroup. By Cauchy's theorem, there exists $g \in G_\mathfrak{p}$ of order $\ell_\mathfrak{p}$ and it follows that for $C_\mathfrak{p} = \innp{g}$ that $C_\mathfrak{p} \cap H_\mathfrak{p} = \set{1}$ for any totally geodesic subgroup. When $\dim(q) = 2n$ and $\mathfrak{p} \in S_1$, we must modify the argument above. If $\det(q)$ is not a square modulo $\mathfrak{p}$, then we can proceed as above since $p^{nr}+1$ divides $\abs{G_\mathfrak{p}}$. When $\det(q)$ is a square modulo $\mathfrak{p}$, we have $G_\mathfrak{p} = SO^+(2n;p^r)$. In this case, there exists $g \in SO^+(2n;p^r)$ such that $g$ has $n/2$ eigenvalues $\lambda_{p^r}$ and $n/2$ eigenvalues $\lambda_{p^r}^{-1}$ where $\lambda_{p^r} \in \BB{F}_{p^r}^\times$ is a generator for the group of units; we can take a generator for the diagonal subgroup of $(SO^+(2,p^r))^n$. Taking $\ell$ to be an odd prime divisor of $p^r-1$, which exists by Theorem \ref{ZsigCor}, and setting $a = (p^r-1)/\ell$, we assert that $C_\mathfrak{p} = \innp{g^a}$ is the desired subgroup. To see this, note that if $\gamma \in \B{PSO}_0(2n-2,1)$, then $\gamma$ has an eigenvalue of $\pm 1$ since $2n-2$ is even. As every totally geodesic $m'$--suborbifold with $m' \geq 2$ in a standard arithmetic orbifold is contained in a codimension one totally geodesic suborbifold (cf.\ \cite{Meyer}), it follows that $\rho_\mathfrak{p}(\gamma)$ has $\pm 1$ as an eigenvalue. As no non-trivial element of $C_\mathfrak{p}$ has this property, $C_\mathfrak{p} \cap H_\mathfrak{p} = \set{1}$. 
\end{proof}

\begin{proof}[Proof of Theorem \ref{thmA} for Standard Arithmetic Orbifolds]
As Theorem \ref{thmA} for $m=3$ was proven in \cite{McR}, we will assume $m\geq 4$ and so $\dim(q) \geq 5$.  We first prove (b). By definition, $\Gamma = \pi_1(M)$ is commensurable with $\mathbf{G}(\Cal{O}_k)$ associated with some admissible hyperbolic pair $(k,q)$.  Strong approximation implies that $\rho_\mathfrak{p}(\Gamma)=\rho_\mathfrak{p}(\mathbf{G}(\Cal{O}_k))$ for all but finitely many $\mathfrak{p}$, hence by Proposition \ref{P:Bob} there is an infinite subset $S_2$ of $S_1$ such that $[G_\mathfrak{p}:\rho_\mathfrak{p}(\Gamma)] = 1$ or $2$ for each $\mathfrak{p} \in S_2$.
By Lemma \ref{L:PrimeCor}, there exists a subgroup $C_\mathfrak{p} < G_\mathfrak{p}$ such that $C_\mathfrak{p} \cap H_\mathfrak{p} = \set{1}$. Since $C_\mathfrak{p}$ is cyclic and of odd prime order, it follows that $C_\mathfrak{p} < \rho_\mathfrak{p}(\Gamma)$. The subgroups $C_\mathfrak{p},\set{1}$ satisfy \eqref{SubgroupCondition} and so the covers $M_1,M_{C_\mathfrak{p}}$ corresponding to the finite index subgroups $\ker \rho_\mathfrak{p}$, $\rho_\mathfrak{p}^{-1}(C_\mathfrak{p})$ are geometrically equivalent.

To produce geometrically equivalent covers with unbounded volume ratio, for each odd prime $\ell$, we set $S_\ell$ to be the subset of primes $\mathfrak{p} \in S_2$ such that $C_\mathfrak{p}$ has order $\ell$. We first assume that $S_\ell$ is infinite for some $\ell$. In that case, for each $j \in \N$ and for any $\mathfrak{p}_1,\dots,\mathfrak{p}_j \in S_\ell$, the image of $\pi_1(M)$ under reduction modulo $\prod_i \mathfrak{p}_i$ has index $2^{s_j}$ in $\prod_i G_{\mathfrak{p}_i}$ for some $s_j \in \N$. By our choice of $\ell$, the subgroup $C_{j,\ell}=\prod_i C_{\mathfrak{p}_i} < \prod_i G_{\mathfrak{p}_i}$ has trivial intersection with the image of any totally geodesic subgroup, and visibly this property holds for any subgroup of $C_{j,\ell}$. Setting $M_j,N_j$ to be the finite covers of $M$ corresponding to the finite index subgroups $\rho_{\mathfrak{p}_1\dots\mathfrak{p}_j}^{-1}(1),\rho_{\mathfrak{p}_1\dots\mathfrak{p}_j}^{-1}(C_{j,\ell})$ of $\Gamma$, we obtain a pair of geometric equivalent finite covers of $M$ with volume ratio $\mathrm{Vol}(M_j)/\mathrm{Vol}(N_j) = \ell^j$.

We now assume that $\abs{S_\ell}$ is finite for all odd primes $\ell$. Since $S_2$ is infinite and each prime $\mathfrak{p} \in S_2$ is in $S_\ell$ for some odd prime $\ell$, there must be infinitely many odd primes $\ell$ with $S_\ell \ne \emptyset$. Fixing an infinite sequence $\set{\ell_j}$ of distinct odd primes with $S_{\ell_j} \ne \emptyset$, for any $j$ and any $\mathfrak{p}_j \in S_{\ell_j}$, we again have $[G_{\mathfrak{p}_j}:\rho_{\mathfrak{p}_j}(\Gamma)]=1$ or $2$. By our choice of $\mathfrak{p}_j$, we have a subgroup $C_{\mathfrak{p}_j} < G_{\mathfrak{p}_j}$ that intersects the image of every totally geodesic subgroup trivially. Setting the manifolds $M_j,N_j$ to be the finite covers of $M$ corresponding to the finite index subgroups $\rho_{\mathfrak{p}_j}^{-1}(1),\rho_{\mathfrak{p}_j}^{-1}(C_{\mathfrak{p}_j})$ of $\Gamma$, we obtain geometrically equivalent finite covers with volume ratio $\ell_j$.

We now prove (a). As $M$ is compact and $\dim(q) \geq 5$, we see that $k \ne \Q$ by Godement's Compactness Criterion (see \cite[Cor 5.3.2]{Witte}) and Meyer's Theorem (see \cite[Prop 6.4.1]{Witte}). Since $k \neq \Q$, by the Cebotarev Density Theorem there is a prime $p$ with two overlying primes $\mathfrak{p}_1,\mathfrak{p}_2 \in S_2$ such that $\mathcal{O}_k/\mathfrak{p}_1 \cong \mathcal{O}_k/\mathfrak{p}_2$. For a pair of such primes $\mathfrak{p}_1,\mathfrak{p}_2$ we have $G_{\mathfrak{p}_1} \cong G_{\mathfrak{p}_2}$, and can apply Lemma \ref{L:PrimeCor} to both. We obtain finite index subgroups $\rho_{\mathfrak{p}_1\mathfrak{p}_2}^{-1}(C_{\mathfrak{p}_1} \times \set{1}),\rho_{\mathfrak{p}_1\mathfrak{p}_2}^{-1}(\set{1} \times C_{\mathfrak{p}_2})$ of $\Gamma$. The associated finite covers $M_{\ell,1},M_{1,\ell}$ of $M$ have the same geometric spectra. To see that $\mathcal{TG}(M_{1,\ell}) = \mathcal{TG}(M_{\ell,1})$, we first note that the finite cover $M_{\ell,\ell}$ associated with the finite index subgroup $\rho_{\mathfrak{p}_1\mathfrak{p}_2}^{-1}(C_{\mathfrak{p}_1} \times C_{\mathfrak{p}_2})$ in $\pi_1(M)$ is geometrically equivalent to both $M_{\ell,1},M_{1,\ell}$ and so $\mathrm{TG}(M_{\ell,1}) = \mathrm{TG}(M_{1,\ell})$. To see that the multiplicities are equal simply note that both manifolds are cyclic covers of $M_{\ell,\ell}$ of degree $\ell$ and thus separately satisfy \eqref{SpectrumRelationship} with $M_{\ell,\ell}$. That the manifolds are nonisometric follows from a similar argument used in \cite[p.~179]{McR}. Briefly, each element $\gamma \in \pi_1(M_{1,\ell})$ is trivial under reduction modulo $\mathfrak{p}_1$ while there are infinitely many elements in $\pi_1(M_{\ell,1})$ with image that generates $C_{\mathfrak{p}_1}$. Consequently, these elements in $\pi_1(M_{\ell,1})$ with order $\ell$ image under modulo $\mathfrak{p}_1$ cannot be conjugate to any element in $\pi_1(M_{1,\ell})$ in $\Isom(\mathbb{H}^m)$. However, if $M_{1,\ell},M_{\ell,1}$ are isometric, by Mostow rigidity, $\pi_1(M_{1,\ell}),\pi_1(M_{\ell,1})$ are conjugate in $\Isom(\mathbb{H}^m)$, and so $M_{1,\ell},M_{\ell,1}$ are nonisometric.
\end{proof}

The proof for a nonstandard arithmetic hyperbolic orbifold $M=\Gamma\backslash\mathbb{H}^m$ is similar.
As in the standard arithmetic setting, there is an associated number field $k$ and an algebraic $k$--group $\mathbf{G}$ for which $\Gamma$ is commensurable with the group $\mathbf{G}(\mathcal{O}_k)$.
There is an infinite set of primes $S_0'$ of $\mathcal{O}_k$ such that for each $\mathfrak{p}\in S_0'$, the local group $\mathbf{G}(k_{\mathfrak{p}})$ is isomorphic to $\mathbf{SO}(V_{\mathfrak{p}},q_{\mathfrak{p}})$, where $(V_{\mathfrak{p}},q_{\mathfrak{p}})$ is a quadratic space over $k_{\mathfrak{p}}$. Restricting to primes in $S_0'$, the proof then follows as in the standard arithmetic case. For (a), we note that when $M$ is a closed arithmetic hyperbolic $m$--orbifold with $m \geq 4$, the field of definition of $M$ is not $\Q$ (see \cite[\S 6.4]{Witte}). 

The above method can be implemented for any finite volume, complete, hyperbolic $m$--orbifold when $m \geq 4$. 

\begin{thm}\label{thmA-Gen}
If $M$ is a complete, orientable, finite volume hyperbolic $m$--orbifold with $m \geq 4$, then the following holds:
\begin{itemize}
\item[(a)]
If the field of definition of $M$ is not $\Q$, then there exist finite, nonisometric covers $M',N'$ such that $M',N'$ are geometrically isospectral.
\item[(b)]
There exists a sequence $(M_j,N_j)$ of pairs of nonisometric finite covers of $M$ such that $M_j,N_j$ is geometrically equivalent and $\mathrm{Vol}(M_j)/\mathrm{Vol}(N_j)$ is unbounded as a function of $j$.
\end{itemize}
\end{thm}

\begin{proof}
Given $M$ with $\Gamma = \pi_1(M)$, there exists an injective homomorphism $\rho\colon \Gamma \to \mathbf{PSO}_0(m,1)$ such that the field generated by the matrix coefficients is a number field $k$ (see \cite{Vinberg} or \cite[\S 4.1]{LongReid}); this field is the so-called field of definition. If $R$ is the $\mathcal{O}_k$--submodule of $k$ generated by the entries of $\rho(\Gamma)$, there is a cofinite subset of the set of prime ideals $\mathcal{P}$ of $\mathcal{O}_k$ such that $R/\mathfrak{P} \cong \mathcal{O}_k/\mathfrak{p} = \BB{F}_{p^r}$ for each $\mathfrak{p} \in \mathcal{P}$, where $\mathfrak{P} = R\mathfrak{p}$. Since $\rho(\Gamma) < \mathbf{PSO}_0(m,1)$ is Zariski dense, we can apply Nori--Weisfeiler strong approximation \cite{Nori,Weisfeiler}. When $m+1$ is odd (resp.\ even), there exists an infinite set of nondyadic primes $S_2 \su \mathcal{P}$ such that the image of $\rho_\mathfrak{P}(\Gamma)$ contains the commutator subgroup $\Omega(m+1; p^r)$ (resp.\ $\Omega^{\pm}(m+1; p^r)$) of $SO(m+1;p^r)$ (resp. $SO^{\pm}(m+1;p^r)$) for each $\mathfrak{P} \in S_2$ (see \cite[Thm.\ 5.3]{LongReid}). The argument now follows as in the previous case of standard arithmetic hyperbolic $m$--orbifolds.
\end{proof}

\medskip

\textbf{Remark 1. }\hypertarget{Rem4} Our use of Zsigmondy's Theorem was inspired by \cite{LongReid}, where they proved that any lattice $\Gamma < \B{SO}(n,1)$ contains hyperbolic elements with infinite order holonomy. In \cite{McR}, the use of Zsigmondy's theorem was replaced by a direct argument. Prasad--Rapinchuk \cite{PrasadRapinchuk} have general results on the existence of semisimple elements whose Zariski closure is dense in a maximal torus. It is possible to replace our elementary counting argument with an argument based on \cite{PrasadRapinchuk}, though one must still determine the possible images of subgroups associated with totally geodesic submanifolds as in \S \ref{sec:Step2}.



\end{document}